\documentclass[12pt, reqno]{amsart}

\usepackage{mymacros,oitp}
\usepackage{tikz-cd}

\usepackage[all]{xy}
\usepackage{fullpage}
\usepackage[pagebackref]{hyperref}

\title{A compact moduli of orbifold projective curves}

\author{Tarig Abdelgadir}
\address{Mathematical Sciences \\
Loughborough University,
LE11 3TU,
United Kingdom
}
\email{t.abdelgadir@lboro.ac.uk}

\author{Daniel Chan}
\address{School of Mathematics and Statistics, 
UNSW Sydney, 
NSW 2052,
Australia
}
\email{danielc@unsw.edu.au}

\author{Shinnosuke Okawa}
\address{
Department of Mathematics,
Graduate School of Science,
Osaka University,
Machikaneyama 1-1,
Toyonaka,
Osaka,
560-0043,
Japan.
}
\email{okawa@math.sci.osaka-u.ac.jp}

\author[K.~Ueda]{Kazushi Ueda}
\address{
Graduate School of Mathematical Sciences,
The University of Tokyo,
3-8-1 Komaba,
Meguro-ku,
Tokyo,
153-8914,
Japan.}
\email{kazushi@ms.u-tokyo.ac.jp}

\begin{document}

\begin{abstract}
We introduce the notion of stable orbifold projective curves,
and show that the moduli stack of stable orbifold projective curves
is isomorphic to the moduli stack of weighted pointed stable curves
in the sense of Hassett
with respect to the weights
determined by the automorphism groups
of the stacky points.
\end{abstract}

\maketitle


\section{Introduction} \label{sc:introduction}

The problem of moduli of curves has a long history
which goes back at least to Riemann.
The set of isomorphism classes of compact Riemann surfaces
can be given a natural complex (or even algebraic) structure
to form the moduli space.
The notion of orbifold is introduced by Satake
(under the name `V-manifold';
the word `orbifold' was invented by Thurston)
in order to formulate the precise sense
in which the moduli spaces
(or more generally quotients of manifolds
by properly discontinuous group actions,
of which the moduli space of curves is a prototypical example)
are `spaces'.
In algebraic geometry,
the notion of stacks is introduced
by Grothendieck
and developed by Artin
to describe such structures.
Deligne and Mumford
\cite{MR262240}
showed
that the moduli stack $\cM_g$ of smooth projective curves
of genus $g \ge 2$
belongs to a particularly nice class of algebraic stacks,
which are now called Deligne--Mumford stacks.

The stack $\cM_g$ is
a non-proper open substack
of the unbounded non-separated moduli stack
$\cU_g$ of proper connected reduced one-dimensional schemes
of arithmetic genus $g$,
and
a proper open substack of $\cU_g$
containing $\cM_g$ as an open dense substack
is called a \emph{modular compactification} of $\cM_g$.
Deligne and Mumford gave a modular compactification $\cMbar_g$,
which is canonical
from the point of view of the minimal model program;
a curve classified by $\cMbar_g$ has at worst nodal singularities
and a finite automorphism group,
implying ampleness of the dualizing sheaf.

In this paper,
we study the moduli stack of one-dimensional stacks.
To define ampleness of an invertible sheaf on a stack,
we use a closed immersion
to a \emph{weighted projective space},
defined  as
the quotient stack of the complement of the origin
of the affine space
by a positive $\Gm$-action.
The automorphism group
of any point
of such a stack
is
the group scheme $\bmu_r$
of $r$-th roots of unity.
We also assume that the automorphism group of the generic point
of any irreducible component is trivial.
Since $\bmu_r$ is not smooth
in characteristics dividing $r$,
a closed substack of a weighted projective space
is not necessarily Deligne--Mumford,
but always \emph{tame}
in the sense of \cite{MR2427954}.
The idea of using embeddings to weighted projective spaces
to study moduli of stacks
goes back at least to \cite{MR2779465},
although the main focus of \cite{MR2779465}
is on varieties of general type
and stacks appear as the stacky Proj of their canonical rings,
whereas we consider stacks which do not come from the canonical rings of varieties.

A one-dimensional closed substack
of a weighted projective space
whose automorphism group at the generic point
of any irreducible component
is trivial will be called
an \emph{orbifold projective curve}
in this paper.
A smooth orbifold projective curve $\bX$ has a numerical invariant
$(g,\bsr)$ consisting of the genus $g$
of the coarse moduli scheme
and the multiset
\begin{align}
  \bsr= \lc \overbrace{ { r }_{ 1 }, \dots, { r }_{ 1 } } ^{ d _{ 1 } },
  \overbrace{ { r }_{ 2 }, \dots, { r }_{ 2 } } ^{ d _{ 2 } },
  \dots,
  \overbrace{ { r }_{ \ell }, \dots, { r }_{ \ell } } ^{ d _{ \ell } }
  \rc
\end{align}
consisting of the orders $r_i$ 
of the automorphism groups of the stacky points.
If $X$ denotes the coarse moduli scheme of $\bX$,
then for $m$ divisible by all the $r_i$,
the sheaf $\omega_{\bX}^{\otimes m}$ is the pullback of $\omega_X^{\otimes m}(m\Delta)$,
where $\Delta$ is the divisor supported at stacky points with standard coefficients $1-r_i^{-1}$.
This suggests considering the pair $(X,\Delta)$
as a weighted pointed curve
in the sense of Hassett \cite{MR1957831},
the weights here being the coefficients of $\Delta$.

Pursuing this connection,
we start with the ``non-weighted'' case
where $\Delta$ is replaced with its round-up $\lceil \Delta \rceil$.
The moduli stack $\cM_{g,\bsr}$ of smooth orbifold projective curves
with the given numerical invariant
can naturally be identified
with the quotient
$
\ld \cM_{g,n} \middle/ \frakS_{\bsd} \rd
$
where
$
\bsd = (d_1,\ldots,d_\ell),
$
$
n=|\bsd| \coloneqq d_1 + \cdots + d_\ell,
$
and
$
\frakS_\bsd = \prod_{i=1}^\ell \frakS_{d_i}.
$
An obvious compactification of
$
\ld \cM_{g,n} \middle/ \frakS_{\bsd} \rd
$
is given by
$
\ld \cMbar_{g,n} \middle/ \frakS_{\bsd} \rd
$.
The universal family
$
\lb
\cC
,
\lb
s_i
\rb_{i=1}^n \rb
$
of stable $n$-pointed curves of genus $g$
produces a family
of orbifold projective curves
over $\cMbar_{g,n}$
by the root construction along the images of the sections $s_i$,
which descends to a family
$
\cX \to
\ld \cMbar_{g,n} \middle/ \frakS_\bsd \rd
$
over the quotient.
By taking the stacky Proj
of the relative canonical ring,
one obtains a family of the stacky canonical models
of the universal family
over
$
\ld \cMbar_{g,n} \middle/ \frakS_\bsd \rd.
$

In
\pref{df:stable orbifold projective curve},
we introduce the notion of a
\emph{stable orbifold projective curve}.
The main result of this paper is
\pref{th:main},
which identifies
the moduli stack of stable orbifold projective curves
with (an unordered variant of)
Hassett's moduli stack of weighted pointed stable curves $(X,\Delta)$,
so that the family of stacky canonical models
constructed above
is the universal family
of stable orbifold projective curves.
This is parallel to the behavior of Hassett's moduli stacks
under the reduction of coefficients
from $\lceil \Delta \rceil$ to $\Delta$. 
It is an interesting problem
to find other modular compactifications,
possibly from the point of view
of minimal model theory for the coarse moduli spaces
(cf.~e.g.~\cite{MR3184168} and references therein).

This paper was originally conceived as a part of \cite{ACOU2},
which bubbled off as a `purely commutative' irreducible component.
The result of \cite{ACOU2} gives yet another modular interpretation
to $\cM_{g,\bsr}$
in such a way that when a pair of stacky points collide,
one obtains not a stacky node
but a `noncommutative point'
which is still `smooth'
(i.e., has finite homological dimension).

\subsection*{Acknowledgements}
During the preparation of this work,
D.~C.~was partially supported by the Australian Research Council Discovery Project grant DP220102861, 
S.~O.~was partially supported by
JSPS Grants-in-Aid for Scientific Research
(16H05994,
16H02141,
16H06337,
18H01120,
20H01797,
20H01794),
and
K.~U.~was partially supported by
JSPS Grants-in-Aid for Scientific Research
(21K18575).

\section{Ample invertible sheaves on algebraic stacks} \label{sc:ample}

\subsection{Assumptions and notations}

Let
$
\pi \colon \bX \to \basescheme
$
be an algebraic stack,
which is
flat,
proper,
and
locally of finite type
over a locally Noetherian scheme
$
\basescheme.
$
Assume that the inertia group stack
is finite,
so that there exists
the coarse moduli algebraic space morphism
$
c \colon \bX \to X
$
by \cite{MR1432041}.
The stack $\bX$ is said to be \emph{cyclotomic}
if the automorphism group of any geometric point
is isomorphic to $\bmu_r$
for some $r$
\cite[Definition 2.3.1]{MR2779465}.
The stack $\bX$ is said to be \emph{tame}
if the functor
$
c_* \colon \Qcoh \bX \to \Qcoh X
$
is exact
\cite[Definition 3.1]{MR2427954}.
This is the case if and only if
the automorphism group scheme
of any geometric point of $\bX$
is finite and linearly reductive
by \cite[Theorem~3.2]{MR2427954}.
Let
$
\pibar \colon X \to B
$
be the morphism
factoring $\pi$ as
$
\pibar \circ c.
$
Let further $L$ be an invertible sheaf on $\bX$.
We will keep these assumptions and notations
through the rest of this section.

\subsection{Definition of ampleness}

In this paper,
we use the following terminology:

\begin{definition}\label{df:coarsely ample}
An invertible sheaf $L$ is \emph{coarsely ample} if
there exists
a positive integer $N$
such that $L^{\otimes N}$ descends
to a $\pibar$-ample invertible sheaf on $X$.
\end{definition}

\begin{remark} \label{rm:existence of descent}
If
the automorphism group of any geometric point
is isomorphic to $ \bmu _{ r }$
for some $ r $
bounded on $\bX$
(which is the case if $\basescheme$ is Noetherian),
then for a sufficiently divisible $ N > 0 $,
the $N$-th tensor power of any invertible sheaf $L$
descends to $X$.
\end{remark}

Let
$
\Sym(L) \coloneqq \bigoplus_{i=0}^\infty L^{\otimes i}
$
be the symmetric algebra,
which is a sheaf of $\bN$-graded $\cO_\bX$-algebras.
The relative spectrum
$
\bL \coloneqq \cSpec \Sym(L)
$
over $\bX$
is the total space of the line bundle on $\bX$ associated with $L$.
The $\bN$-grading on $\Sym(L)$
gives the $\Gm$-action on $\bL$
by fiberwise dilation.

\begin{lemma}
If an invertible sheaf $L$ is coarsely ample,
then the sheaf
$\pi_* \Sym(L)$
of $\bN$-graded $\cO_B$-algebras
is of finite type.
\end{lemma}

The relative spectrum
$
\cSpec \pi_* \Sym(L) 
$
is the affinization of $\bL$ relative to $B$.
The structure morphism
$
\bL \to \cSpec \pi_* \Sym(L)
$
is $\Gm$-equivariant
and hence defines a morphism
\begin{align}
  \phi_L
  \colon
  \ld \bL \middle/ \Gm \rd
  \to
  \ld \cSpec \pi_* \Sym(L) \middle/ \Gm \rd
\end{align}
of quotient stacks.
Recall that
for an $\bN$-graded $\cO_B$-algebra
$
\cA = \bigoplus_{i=0}^\infty \cA_i,
$
the stack
\begin{align}
  \bProj \cA
  \coloneqq
  \ld \lb \cSpec \cA \setminus \bszero_{\cSpec \cA} \rb \middle/ \Gm \rd
\end{align}
is the quotient of the complement
the closed subscheme
$
\bszero_{\cSpec \cA}
$
defined by the irrelevant ideal
$
\cA_+ \coloneqq \bigoplus_{i=1}^\infty \cA_i.
$
The coarse ampleness of $L$ implies that
the restriction of $\phi_L$
to the complement of 
(the $\Gm$-quotient of)
the zero section
$\bszero_{\bL}$
factors as the composite
$
\iota \circ \phi_L
$
of a morphism
\begin{align} \label{eq:varphiL}
  \varphi_L
  \colon
  \bX
  \simeq
  \ld \lb \bL \setminus \bszero_\bL \rb \middle/ \Gm \rd
  \to
  \bProj \pi_* \Sym(L)
\end{align}  
and the open immersion
\begin{align}
  \iota
  \colon
  \bProj \pi_* \Sym(L)
  \hookrightarrow
  \ld \cSpec \pi_* \Sym(L) \middle/ \Gm \rd.
\end{align}

\begin{definition} \label{df:ample}
An invertible sheaf $L$ is \emph{ample}
if it is coarsely ample and
the morphism \eqref{eq:varphiL}
is an isomorphism.
\end{definition}

\begin{remark}
A stack is cyclotomic
if it admits an ample invertible sheaf.
\end{remark}

\subsection{Ampleness criterion after Abramovich and Hassett}

\begin{definition}[{\cite[Definition~2.4.1]{MR2779465}, cf.~also \cite[Definition~2.7]{MR2819757}}] \label{df:orbi-ample}
An invertible sheaf $L$
on a cyclotomic stack $\bX$
is \emph{polarizing}
if
\begin{itemize}
  \item 
  $L$ is coarsely ample, and
  \item
  for any geometric point $x$ of $\bX$,
  the character associated with the stalk $L_x$
  generates
  the finite cyclic group
  $
  \Hom \lb \Aut x, \Gm \rb.
  $
\end{itemize}
\end{definition}

\pref{pr:ampleness criterion} below
is a slight variant of
\cite[Corollary 2.4.4]{MR2779465}
(cf.~also \cite[Proposition~2.11]{MR2819757}),
whose proof we include for the sake of completeness.
  
\begin{proposition} \label{pr:ampleness criterion}
An invertible sheaf $L$ on a cyclotomic stack $\bX$ is ample
if and only it is polarizing.
\end{proposition}

\begin{proof}
Take a very ample invertible sheaf $\Lbar$
on the coarse moduli space $ X $ and $ d > 0 $
such that $ L ^{ \otimes d } \simeq c ^{ \ast } \Lbar $.
Then one has
\begin{align}
X
\simeq \Proj R \lb \Lbar \rb
\simeq \Proj R \lb L^{\otimes d} \rb
\simeq \Proj R \lb L \rb
\end{align}
where
$
R \lb \Lbar \rb \coloneqq \pibar_* \Sym \lb \Lbar \rb
$
and so on,
so that $\bX$ and $\bProj R (L)$
have isomorphic coarse moduli spaces and
we have the commutative diagram:
\begin{equation}
\begin{tikzcd}
\bX \arrow[r, "\varphi _{ L }"] \arrow[d,"c"'] & \bProj R ( L ) \arrow[d,"\pi"]\\
X \arrow[r,"\simeq"] & \Proj R ( L )
\end{tikzcd}
\end{equation}

By \cite[Theorem~3.2]{MR2427954},
for any closed point $ x \in X $,
there is an \'etale neighborhood $ V \to X $ such that
$
\bX \times _{ X } V
$
is isomorphic to the quotient stack
$
\left[ W / G \right]
$
for some finite $V$-scheme $ W $ and
a finite linearly reductive $ V $-group scheme $ G $ acting on $ W $.
Let $ v \in V $ be a point mapping to $ x $, and let $ A $ be the strict henselization of the local ring
$
\cO_v.
$
In order to prove that
$\varphi_L$ is an isomorphism,
it suffices to prove that
the base change $\lb \varphi_L \rb_A$
of $ \varphi _{ L } $ by $ \Spec A \to X $ is an isomorphism.

We write
$
\Spec S = W \times_V \Spec A
$
and
$
\bX_A = \bX \times_V \Spec A.
$
By our assumption
on the automorphism group scheme of geometric points of $ \bX $
and \cite[Lemma 2.17]{MR2427954},
we have
$
G \times_V \Spec A
\cong \bmu_r
$
for some positive integer $r$.
By our assumption on $L$,
there is a generator $\chi$ of $\Hom(\bmu_r, \Gm)$
such that
$
L|_{\bX_A}
$
is identified with the $\Hom(\bmu_r, \Gm)$-graded $ S $-module $ S ( \chi ) $.
We may choose an isomorphism
$
\Hom(\bmu_r, \Gm) \simeq \bZ / r \bZ
$
so that $ \chi $ corresponds to $ [1] \in \bZ / r \bZ$.
Then
$
S = \bigoplus_{i=0}^{r-1} S_{[i]}
$
is a $\bZ / r \bZ$-graded algebra and
$
L \vert _{ \bX_A }
$
is identified with the graded $S$-module $ S ([1])$.
We have
$
 R \coloneqq R \lb L|_{\bX_A} \rb
 = \bigoplus_{i=0}^\infty R_i
$
with
$
R_i = S_{[i]},
$
so that
\begin{align}
  \Spec R \simeq \lb \Spec S \times \bA^1 \rb / \bmu_r
\end{align}
where
$\bmu_r$ acts on $\bA^1$ through the character $\chi$.
Since the restriction of this action of $\bmu_r$ to
$
\Gm \simeq \bA^1 \setminus \bszero_{\bA^1}
$
is free
and
$
\bszero_{\Spec R} \simeq \Spec S \times \bszero_{\bA^1},
$
we have an isomorphism
\begin{align}
\bProj R
&\coloneqq \ld \lb \Spec R \setminus \bszero_{\Spec R} \rb \middle/ \Gm \rd \\
&\simeq \ld \lb \Spec S \times \Gm \rb \middle/ \lb \bmu_r \times \Gm \rb \rd \\
&\simto \ld \Spec S \middle/ \bmu_r \rd \\
&\simeq \bX_A,
\end{align}
which is inverse to
$
\lb \varphi_L \rb_A.
$
\end{proof}

\subsection{Ampleness criterion after Artin and Zhang}

We give \pref{th:AZ-ample} below
to justify \pref{df:ample};
it will not be used in the proof of \pref{th:main}.

\begin{definition}[{\cite[(4.2.1)]{MR1304753}}]  \label{df:AZ-ample}
An invertible sheaf $L$ is \emph{AZ-ample}
if
\begin{enumerate}[(a)]
\item
for any
$
\cM \in \coh \bX,
$
there exist
$
p \in \bN,
$
$
\ell _{ 1 }, \dots, \ell _{ p } \in \bZ,
$
and an epimorphism
$
\bigoplus _{ i = 1 } ^{ p } L ^{ \otimes ( - \ell _{ i } ) } \to \cM,
$
and
\item
for any epimorphism
$
f \colon \cM \to \cN
$
in $\coh \bX$,
there exists $ n _{ 0 } \in \bZ$
such that
$
\pi _{ \ast } ( \cM \otimes L ^{ \otimes n } )
\to
\pi _{ \ast  }( \cN \otimes L ^{ \otimes n } )
$
is surjective for any $ n \ge n _{ 0 } $.
\end{enumerate}
\end{definition}

\begin{theorem} \label{th:AZ-ample}
An invertible sheaf $L$ is ample
if and only if it is AZ-ample.
\end{theorem}

\begin{proof}
Since the claim is local in $\basescheme$,
we may and will assume that $\basescheme$ is affine.
We set
$
R \coloneqq \pi_* \Sym(L),
$
and identify
\begin{align}
\qgr R
&\coloneqq
\gr R / \tor R \\
&\simeq
\left.
\coh \ld \Spec R \middle/ \Gm \rd
\middle/
\coh_{\ld \bszero / \Gm \rd} \ld \Spec R \middle/ \Gm \rd
\right. \\
&\simeq
\coh \ld \lb \Spec R \setminus \bszero \rb \middle/ \Gm \rd \\
&\simeq \coh \bProj R.
\end{align}

The only if part is an immediate consequence of
\cite[Proposition 3.11.(3)]{MR1304753}
and
\cite[Corollary 4.6.(2)]{MR1304753}.

In order to prove the if part,
assume that $L$ is AZ-ample.
The functor
\begin{align}
\iota^* \circ \lb \phi_L \rb_* \colon \coh \bX \to \coh \bProj R, \ 
F
\mapsto
\iota^* \lb \bigoplus_{i=0}^\infty \pi_* \lb \cF \otimes L^{\otimes i} \rb \rb
\end{align}
is an equivalence
by \cite[Theorem 4.5.(1)]{MR1304753}.
If the image of $\phi_L$ intersects
$
\ld \bszero \middle/ \Gm \rd,
$
then the push-forward of the structure sheaf of
$
\phi_L^{-1} \lb \ld \bszero \middle/ \Gm \rd \rb
\coloneqq
\ld \bszero \middle/ \Gm \rd \times_{\ld \Spec R \middle/ \Gm \rd} \bX
$
gives a non-zero object of $\coh \bX$
which goes to zero by $\iota^* \circ \lb \phi_L \rb_*$,
which contradicts the fact that $\iota^* \circ \lb \phi_L \rb_*$
is an equivalence.
It follows that
$\phi_L$ factors through $\varphi_L$
and $\iota^* \circ \lb \phi_L \rb_* \simeq \lb \varphi_L \rb_*$,
so that the functor $\varphi_L^*$,
which is left adjoint to $\lb \varphi_L \rb_*$,
is an equivalence.
Since the functor $\varphi_L^*$ is monoidal
with respect to the tensor product of coherent sheaves,
the morphism $\varphi_L$ is an isomorphism of stacks
(see \cite[Remark~5.12]{0412266}).
\end{proof}

\section{Stable orbifold projective curves} \label{sc:stable}

The following is the main definition in this paper:

\begin{definition} \label{df:stable orbifold projective curve}
An algebraic stack $ \bX $ over an algebraically closed field $\bfk$
is a \emph{stable orbifold projective curve}
if it is connected,
reduced,
and proper over $ \bfk $
of dimension one
satisfying the following conditions:
\begin{enumerate}[(1)]
\item \label{it:no generic stabilizer}
The open substack of $\bX$ consisting of points
with the trivial automorphism group
is dense.
\item \label{it:nodal}
The singularities of the coarse moduli space $X$ are at-worst-nodal.
\item
If a stacky point of
$
  \bX
$
is not smooth,
then the formal completion of $\bX$ at the point is isomorphic to the quotient stack
\begin{align}\label{eq:stacky singularity}
  \left[ \Spec \frac{ \bfk \db[ x, y \db] }{ ( x ^{ 2 } - y ^{ 2 } ) } \middle/ \bmu _{ 2 } \right],
\end{align}
where the coaction of
$
  \bfk[\bmu _{ 2 }] = \bfk [ t ] / ( t ^{ 2 } - 1 )
$
is given by
$
  x \mapsto x \otimes t
$
and
$
  y \mapsto y \otimes 1.
$
\item
The dualizing sheaf
$
  \omega _{ \bX }
$
is ample.
\end{enumerate}
A stable orbifold projective curve
over a scheme $ \basescheme $
is a proper flat stack
$
  \pi \colon \bX \to \basescheme
$
whose geometric fibers are stable orbifold projective curves.
\end{definition}


The formal completion
of a stable orbifold projective curve over an algebraically closed field
at a stacky point
is isomorphic to
$
\left[ \Spec \bfk \db[ x \db] \middle/ \bmu_r \right]
$
for some positive integer $r$
with the coaction of $\bfk[\bmu_r]$ given by
$x \mapsto x \otimes t$
if the point is smooth,
and
to \eqref{eq:stacky singularity}
if the point is singular.

\begin{remark}
If $ \bfk $ is of characteristic $ 2 $,
then the stack \eqref{eq:stacky singularity} is isomorphic to
\begin{align}
\left[
\frac{ \Spec \bfk \db[ x, y \db] }{ ( x + y ) ^{ 2 } } \bigg/ \bmu _{ 2 }
\right].
\end{align}
The coarse moduli scheme is
$ \Spec \bfk \db[ y \db] $.
One can easily confirm that the action of $ \bmu _{ 2 } $ is free
over the punctured spectrum
$
  \Spec \bfk \db[ y \db] [y^{-1}] \hookrightarrow \Spec \bfk \db[ y \db].
$
By directly computing a smooth atlas, one can also confirm that this is actually a \emph{reduced} algebraic stack.

In any characteristic,
the stacky singularity \eqref{eq:stacky singularity} admits a deformation
to a pair of smooth stacky points with
$
\bmu _{ 2 }
$
automorphism groups.
\end{remark}

\begin{definition} \label{df:moduli}
The \emph{moduli stack of stable orbifold projective curves}
$
\cMbar
$
is the following category
equipped with the obvious forgetful functor
$
\cMbar \to \left( \Sch / \basescheme \right).
$

\begin{enumerate}[(1)]
\item
An object is a stable orbifold projective curve
$
  f \colon \bX \to U
$
over a $\basescheme$-scheme
$
U.
$
\item
A morphism
\begin{align}
  \varphi = \left( \varphi _{ 1 }, \varphi _{ 2 } \right) \colon
  \left( f \colon \bX \to U \right) \to \left( g \colon \bY \to V \right)
\end{align}
consists of a morphism
$
  \varphi_2 \colon U \to V
$
of $\basescheme$-schemes and a natural isomorphism class of morphisms of stacks
$
  \varphi _{ 1 } \colon \bX \to \bY
$
such that the following diagram is cartesian:
\begin{equation}
\begin{tikzcd}
\bX \arrow[r, "\varphi _{ 1 }"] \arrow[d, "f"'] & \bY \arrow[d, "g"] \\
U \arrow[r, "\varphi _{ 2 }"'] & V
\end{tikzcd}
\end{equation}
\end{enumerate}
For $g \in \bN$ and $\bsr \in \lb \bZ^{\ge 2}\rb^n$,
the stack $\cM_{g,\bsr}$
is the substack of $\cMbar$
representing smooth orbifold projective curves with exactly $n$ stacky points
with automorphism groups $\bmu_{r_i}$
for $i=1,\ldots,n$
whose coarse moduli spaces are of genus $g$.
The stack $\cMbar_{g,\bsr}$ is the closure of $\cM_{g,\bsr}$ in $\cMbar$.
\end{definition}

\begin{remark}
We can consider a $ 2 $-category of stable orbifold projective curves,
where $ 2 $-morphisms are natural isomorphisms between morphisms of stable orbifold projective curves.
However, by \cite[Lemma~4.2.3]{MR1862797}
(although it is stated only for Deligne--Mumford stacks,
the proof works equally well for tame algebraic stacks by replacing `\'etale' with `smooth'
everywhere in the proof),
there is at most one natural isomorphism for each pair of morphisms of stacks,
so that the resulting $ 2 $-category is equivalent
to the $ 1 $-category defined in \pref{df:moduli}
(\cite[Proposition~4.2.2]{MR1862797}).
\end{remark}

\section{Weighted pointed stable curves} \label{sc:Hassett}

We recall the work of Hassett in this section.

\begin{definition}[{\cite{MR1957831}}]\label{df:Hassett_moduli}
The \emph{moduli stack of weighted pointed stable curves}
$
\cMbar_{g,\scrA}
$
of genus $g \in \bN$
and weight
$
\scrA = (a_i)_{i=1}^n
\in
\lb (0,1] \cap \bQ \rb^n
$
is the following category
equipped with the obvious forgetting functor
$
\cMbar_{g,\scrA} \to \lb \Sch \middle/ \basescheme \rb.
$
\begin{enumerate}[(1)]
\item
An object
$
\lb f \colon X \to U, s_1,\ldots,s_n \rb
$
consists of
\begin{itemize}
\item
a flat projective morphism
$
 f \colon X \to U
$
of
$\basescheme$-schemes, and
\item
sections
$
 s _{ 1 }, \dots, s _{ n } \colon U \to X
$
of
$
 f
$,
\end{itemize}
such that
\begin{itemize}
\item
any geometric fiber of $f$ is a connected at-worst-nodal curve
of arithmetic genus $g$,
\item
the images of the sections
$
 s _{ 1 }, \dots, s _{ n }
$
are contained in the smooth locus of $f$,
\item
$
 a _{ i _{ 1 } } + \cdots + a _{ i _{ \ell } } \le 1
$
whenever the images of
$
 s _{ i _{ 1 } }, \dots, s _{ i _{ \ell } }
$
have non-empty intersection,
and
\item
the $\bQ$-line bundle
$
 \omega _{ f } \left( a _{ 1 } s _{ 1 } + \cdots + a _{ n } s _{ n } \right)
$
is $f$-ample.
\end{itemize}
\item
A morphism
\begin{align}
 \varphi = \left( \varphi _{ 1 }, \varphi _{ 2 } \right) \colon
 \left( f \colon X \to U, s _{ 1 }, \dots, s _{ n } \right)
 \to
 \left( g \colon Y \to V, t _{ 1 }, \dots, t _{ n } \right)
\end{align}
consists of morphisms of
$
 \bfk
$-schemes
$
 \varphi_2 \colon U \to V
$
and
$
 \varphi _{ 1 } \colon X \to Y
$
such that
$
 s _{ i } \circ \varphi _{ 2 } = \varphi _{ 1 } \circ t _{ i }
$
for all
$
 i = 1, \dots, n
$
and the following diagram is cartesian.
\begin{align}
 \xymatrix{
 X \ar[r] ^{ \varphi _{ 1 } } \ar[d] _{ f } & Y \ar[d] ^{ g } \\
 U \ar[r] _{ \varphi _{ 2 } } & V}
\end{align}
\end{enumerate}
\end{definition}

\begin{remark}
The definition of $\cMbar_{g,\scrA}$ reduces to that of $\cMbar_{g,n}$
when $\scrA = (1, \ldots, 1)$.
\end{remark}

\begin{theorem}[{\cite[Theorem 2.1]{MR1957831}}] \label{th:Hassett}
The category $\cMbar_{g,\scrA}$ is a connected smooth proper Deligne--Mumford stack
over $\basescheme$.
\end{theorem}

\begin{remark} \label{rm:degree_of_the_log_canonical_divisor_is_positive}
The stack $\cMbar_{g,\scrA}$ is empty unless
$
 2 g - 2 + \sum _{ i = 1 } ^{ n } a _{ i } > 0.
$
\end{remark}

\begin{theorem}[{\cite[Theorem 4.1]{MR1957831}}]
If $\scrA=(a_i)_{i=1}^n$ and $\scrB=(b_i)_{i=1}^n$ satisfies
$b_i \le a_i$ for all $i \in \{ 1, \ldots, n \}$,
then there exists a birational morphism
\begin{align} \label{eq:reduction}
  \cMbar_{g,\scrA} \to \cMbar_{g,\scrB}
\end{align}
obtained by contracting components of the curve $C$
along which $\omega_C + \sum_{i=1}^n b_i s_i$ fails to be ample.
\end{theorem}

\begin{theorem}[{\cite[Corollary 4.7]{MR1957831}}] \label{th:reduction isomorphism}
Assume that for any $I = \{ i_1,\ldots,i_r \} \subset \{ 1, \ldots, n \}$
such that $a_{i_1} + \cdots + a_{i_r} > 1$
and $b_{i_1} + \cdots + b_{i_r} \le 1$,
one has $r=2$.
Then the reduction morphism \pref{eq:reduction} is an isomorphism.
\end{theorem}

\pref{cr:Hassett isomorphism} below
is an immediate consequence of \pref{th:reduction isomorphism}:

\begin{corollary} \label{cr:Hassett isomorphism}
If $a_i \ge \frac{1}{2}$ for all $i \in \{ 1, \ldots, n \}$,
then the reduction morphism gives an isomorphism
$
\cMbar_{g,(1,\ldots,1)} \to \cMbar_{g,\scrA}.
$
\end{corollary}

\section{Curves with weighted divisors} \label{sc:KSBA}

A variant $\cNbar_{g,\cA}$
of $\cMbar_{g,\scrA}$,
which we call the \emph{KSBA moduli of 1-dimensional stable pairs}
after Koll\'ar, Shepherd-Barron, and Alexeev,
is defined as follows:

\begin{definition}\label{df:KSBA_moduli}
The \emph{KSBA moduli stack of 1-dimensional stable pairs}
$
\cNbar_{g,\cA}
$
of genus $g \in \bN$
and weight
$
\cA = \lb \lb a_i, d_i \rb \rb_{i=1}^\ell
\in \lb \lb (0,1] \cap \bQ \rb \times \bZ^{>0} \rb^\ell
$
is the following category
equipped with the obvious forgetful functor
$
\cNbar_{g,\cA} \to \lb \Sch \middle/ \basescheme \rb.
$
\begin{enumerate}[(1)]
\item
An object
\begin{align}
  \left( f \colon X \to U, D _{ 1 }, \dots, D _{ \ell } \right)
\end{align}
consists of
\begin{itemize}
\item
a flat projective morphism
$
 f \colon X \to U
$
of $\basescheme$-schemes
with connected geometric fibers
of arithmetic genus $g$, and
\item
relative effective Cartier divisors
$
 D _{ 1 }, \dots, D _{ \ell } \subset X
$
(see e.g.~\cite[\href{https://stacks.math.columbia.edu/tag/056P}{Tag 056P}]{stacks-project}
for the notion of relative effective Cartier divisors)
of degrees
$
 d _{ 1 }, \dots, d _{ \ell },
$
\end{itemize}
such that
\begin{itemize}
\item
any geometric fiber of $f$ together with the restriction of the divisor
$
 a _{ 1 } D _{ 1 } + \cdots + a _{ \ell } D _{ \ell }
$
(which is called a \emph{weighted divisor} in \cite[Section~2.1.3]{MR1957831})
is a semi-log canonical pair of dimension one, and
\item
the $\bQ$-line bundle
$
 \omega _{ f } \left( a _{ 1 } D _{ 1 } + \cdots + a _{ \ell } D _{ \ell } \right)
$
is $f$-ample.
\end{itemize}
\item
A morphism
\begin{align}
 \varphi = \left( \varphi _{ 1 }, \varphi _{ 2 } \right) \colon
 \left( f \colon X \to U, D _{ 1 }, \dots, D _{ \ell } \right)
 \to
 \left( g \colon Y \to V, E _{ 1 }, \dots, E _{ \ell } \right)
\end{align}
consists of morphisms of
$
    \basescheme
$-schemes
$
 \varphi_2 \colon U \to V
$
and
$
 \varphi _{ 1 } \colon X \to Y
$
such that
$
 \varphi _{ 1 } ^{ * } E _{ i } = D _{ i }
$
for all
$
 i = 1, \dots, \ell
$
and the following diagram is cartesian.

\begin{align}
 \xymatrix{
 X \ar[r] ^{ \varphi _{ 1 } } \ar[d] _{ f } & Y \ar[d] ^{ g } \\
 U \ar[r] _{ \varphi _{ 2 } } & V}
\end{align}

\end{enumerate}
\end{definition}

The assumption on the singularity of the geometric fibers $F$ of $f$ is equivalent to the conditions that
$F$ has at worst nodal singularity, the coefficients of the effective divisor
$
 a _{ 1 } D _{ 1 } | _{ F }  + \cdots + a _{ \ell } D _{ \ell } | _{ F }
$
do not exceed $ 1 $, and that
$
 D _{ i }
$
are away from the singularity of $ F $.

The following claim implicitly appears in \cite{MR1957831}.

\begin{theorem} \label{th:Nbar}
There exists an isomorphism
\begin{align}
 \cNbar_{g,\cA} \simeq \ld \cMbar_{g,\scrA} \middle/ \frakS_\bsd \rd
\end{align}
of stacks where
$
 \frakS _{ \bsd } = \prod _{ i = 1 } ^{ \ell } \frakS _{ d _{ i } }
$.
\end{theorem}

\begin{remark}
The stack
$
 \left[ \cMbar _{ g, \scrA } / \frakS _{ \bsd } \right]
$
is defined as the \emph{stackification} of the category fibered in groupoids over
$
 \Sch / \basescheme
$
defined as follows.
An object is a triple consisting of a $\basescheme$-scheme
$
 U
$,
a principal
$
 \frakS _{ \bsd }
$
-bundle
$
    P \to U,
$
and an equivariant morphism
$
    P \to \cMbar _{ g, \scrA }
$.
A morphism between these two objects is an isomorphism of the principal
$
 \frakS _{ \bsd }
$-bundles
which is also compatible with the morphisms to
$
 \cMbar _{ g, \scrA }
$.
In general, this is not a stack before taking stackification.
\end{remark}

Although $\cNbar_{g,\cA}$ is not necessarily tame,
\pref{th:Hassett} and \pref{th:Nbar} implies the following:

\begin{corollary}
The stack $\cNbar_{g,\cA}$ is a connected smooth proper Deligne--Mumford stack
over $\basescheme$.
\end{corollary}

\section{The main theorem} \label{sc:main}

Given
$
 \bsr = (r_1,\ldots,r_n) \in (\bZ^{\ge 2})^n,
$
we define the \emph{standard weight}
associated with $\bsr$ as
\begin{align}
\scrA_\bsr \coloneqq \left( a _{ 1 }, \dots, a _{ n } \right)
\coloneqq
\left( 1 - \frac{1}{ r _{ 1 } }, \dots, 1 - \frac{1}{ r _{ n } } \right),
\end{align}
which we further write as
\begin{align}
  \scrA_\bsr
  =
  \left( \overbrace{ { a }_{ 1 }, \dots, { a }_{ 1 } } ^{ d _{ 1 } },
  \overbrace{ { a }_{ 2 }, \dots, { a }_{ 2 } } ^{ d _{ 2 } },
  \dots,
  \overbrace{ { a }_{ \ell }, \dots, { a }_{ \ell } } ^{ d _{ \ell } }
  \right)
\end{align}
for pairwise distinct numbers
$
a _{ 1 }, \dots, a _{ \ell }
$
possibly after permutation,
and set
\begin{align}\label{eq:cAr}
  \cA_\bsr \coloneqq
  \lb \lb a_1, d_1 \rb,\ldots, \lb a_\ell, d_\ell \rb \rb.
\end{align}


\begin{theorem}\label{th:main}
There exists an isomorphism
$
\cMbar_{g,\bsr}
\simeq
\cNbar_{g,\cA_\bsr}
$
of stacks over $\basescheme$.
\end{theorem}


\begin{proof}
A morphism
\begin{align} \label{eq:N to M}
F \colon \cNbar_{g,\cA_\bsr} \to \cMbar_{g,\bsr}
\end{align}
of stacks
is defined
by sending
a KSBA family
\begin{align} \label{eq:KSBA family}
  \lb f \colon X \to U, D _{ 1 }, \dots, D _{ \ell } \rb
\end{align}
of genus $g$ and weight $\cA_\bsr$
to the root stack
\begin{align}
  \bX
  \coloneqq
  X \ld \sqrt[r_{ 1 }]{D _{ 1 } }, \dots, \sqrt[r_{ \ell }]{D _{ \ell } }\rd.
\end{align}
Stability of the resulting family
follows from \pref{pr:ampleness criterion}
and a straightforward computation
at geometric points of $U$.

To define
a morphism
$
 G \colon \cMbar_{g,\bsr} \to \cNbar_{g,\cA_\bsr}
$
inverse to \pref{eq:N to M},
let
$
 \bX \to U
$
be a stable orbifold projective curve.
Take the coarse moduli scheme
$
 c \colon \bX \to X,
$
and
consider the natural morphism
\begin{align}
    i \colon c ^{ \ast } \omega _{ X } \hookrightarrow \omega _{ \bX },
\end{align}
which is the `pull-back of differential forms'.
Take $ N > 0 $ such that
$
    \omega _{ \bX } ^{ \otimes N } \simeq c ^{ \ast } L
$
for an invertible sheaf $ L $ on $ X $. Then the map
\begin{align}
    c ^{ \ast } \omega _{ X } ^{ \otimes N }
    \stackrel{ i ^{ \otimes N } }{\hookrightarrow}
    \omega _{ \bX } ^{ \otimes N } \simeq c ^{ \ast } L
\end{align}
corresponds to a map
$
    j \colon \omega _{ X } ^{ \otimes N } \hookrightarrow L
$
under the adjunction
$
    c ^{ \ast } \dashv c _{ \ast }
$
and the isomorphism
$
    c _{ \ast } \cO _{ \bX } \simeq \cO _{ X }
$.
Hence there exists an effective Cartier divisor $ \Dtilde $ on $ X $ which is supported in the smooth locus of $ X \to U $ such that there is an isomorphism
\begin{align}\label{eq:canoical bundle formula downstairs}
    L \simeq \omega _{ X } ^{ \otimes N } \lb \Dtilde \rb.
\end{align}
By applying $ c ^{ \ast } $ to \eqref{eq:canoical bundle formula downstairs},
we obtain the isomorphism
\begin{align}\label{eq:canoical bundle formula upstairs}
    \omega _{ \bX } ^{ \ast N }
    \simeq
    c ^{ \ast } \left(  \omega _{ X } ^{ \otimes N } \lb \Dtilde \rb \right)
\end{align}
on $ \bX $,
and
\begin{align} \label{eq:boundary divisor}
    D \coloneqq \frac{ 1 }{ N } \Dtilde
\end{align}
is an effective $ \bQ $-divisor on $ X $ with standard coefficients.
If there is a component $ C $ of $ D $ with coefficient $ 1 $, we think of it as the sum of two copies of
$
    \frac{ 1 }{ 2 } C
$.
By \eqref{eq:canoical bundle formula upstairs}, the pair
$
( X, D )
$
is log crepant to $ \bX $. Hence the fibers of \eqref{eq:(X,D) -> U} are semi-log canonical and the log canonical $ \bQ $-divisor $ \omega _{ X } ( D ) $ is ample over $ U $.
It follows that
\begin{align}\label{eq:(X,D) -> U}
f \colon ( X , D ) \to U
\end{align}
is a KSBA family over $ U $.

It is clear from the construction that
$F$ and $G$ are inverse to each other,
and \pref{th:main} is proved.
\end{proof}

\begin{remark}
The inverse of the reduction isomorphism
\begin{align}
\cNbar _{ g, \left( ( 1, d _{ 1 } ), \dots, ( 1, d _{ \ell } ) \right) }
\simto
\cNbar _{ g, \cA }
\end{align}
in \pref{cr:Hassett isomorphism}
gives a KSBA-family
\begin{align}\label{eq:inflated model}
\left( \ftilde \colon \Xtilde \to \cNbar_{g,\cA}, \Dtilde _{ 1 }, \dots, \Dtilde _{ \ell } \right)
\end{align}
of genus $ g $ and weight
$
\left( ( 1, d _{ 1 } ), \dots, ( 1, d _{ \ell } ) \right)
$
over $\cNbar_{g,\cA}$.
The difference between the original KSBA family and \pref{eq:inflated model}
is that a pair of points with weight $\frac{1}{2}$ can collide in the KSBA family,
which is prohibited
in the inflated family
and
$\bP^1$ with two distinct points bubbles off instead.
The reduction morphism $\Xtilde \to X$ contracts
those components.
The \emph{canonical model}
\begin{align} \label{eq:canonical model}
f_\can \colon \bXtilde _{ \can }
\coloneqq \bProj  \lb \bigoplus_{i=0}^\infty f_* \lb \omega_{\bXtilde/U}^{\otimes i} \rb \rb
\to U
\end{align}
of the relative canonical ring
of the root stack
\begin{align}
\bXtilde
\coloneqq \Xtilde \ld \sqrt[r_{ 1 }]{\Dtilde _{ 1 } }, \dots, \sqrt[r_{ \ell }]{\Dtilde _{ \ell } }\rd
\end{align}
gives a family of the stable orbifold projective curves,
whose classifying morphism
$
\cNbar_{g,\cA_\bsr} \to \cMbar_{g,\bsr}
$
gives the isomorphism \pref{eq:N to M}.
\end{remark}

\bibliographystyle{amsalpha}
\bibliography{bibs}

\end{document}